\documentclass[preprint,12pt]{elsarticle}

\usepackage[english]{babel}
\usepackage[english]{babel}
\usepackage{amsfonts}
\usepackage{amssymb, amsthm}
\usepackage{xcolor}
\usepackage{marginnote}
\usepackage{amsmath}
\usepackage{a4wide}
\numberwithin{equation}{section}

\usepackage{accents}

\usepackage{amssymb}
\usepackage{amsthm}

\newtheorem{thm}{Theorem}
\newtheorem{lemma}{Lemma}
\newtheorem{proposition}{Proposition}

\theoremstyle{definition}

\journal{Nonlinear Analysis: Real World \& Applications}

\begin{document}

\begin{frontmatter}



\title{On nonlinear boundary value problem corresponding to\\  $N$-dimensional inverse spectral problem }

\author[rvt]{Y.Sh.~Ilyasov\corref{cor1}}
\ead{ilyasov02@gmail.com}

\author[rvt]{N. F.~Valeev}
\ead{valeevnf@mail.ru}

\cortext[cor1]{Corresponding author}

\address[rvt]{Institute of Mathematics of UFRC RAS, 112, Chernyshevsky str., 450008 Ufa, Russia}

\begin{abstract}

We establish a  relationship between an inverse optimization spectral problem
for N-dimensional Schr\"odinger equation $
-\Delta \psi+q\psi=\lambda \psi
$ and a solution of the nonlinear
boundary value problem $-\Delta u+q_0 u=\lambda u- u^{\gamma-1},~~u>0,~~
	u|_{\partial \Omega}=0$. Using this relationship, we find an
exact solution for the inverse optimization spectral problem, investigate its stability and
obtain new results on the existence and uniqueness of the solution for the nonlinear boundary
value problem.

\end{abstract}


\end{frontmatter}


\section{Introduction}
This paper is concerned with the inverse spectral problem for the operator of the form
\begin{equation} \label{eq:S}
\mathcal{L}_q\phi:=-\Delta \phi+q\phi,~~~ x\in \Omega,
\end{equation}
subject to the Dirichlet boundary condition
\begin{equation} \label{eq:Sq}
	\phi \bigr{|}_{\partial \Omega}=0.
\end{equation}
Here $\Omega$ is a bounded domain in $\mathbb{R}^N$, $N\geq 1$, the boundary $\partial \Omega$ is of class $C^{1,1}$.
We assume that  $q \in L^p(\Omega)$, where 
\begin{equation}\label{Pas}
p \in 	\begin{cases}
	[2,+\infty)~~~\mbox{if}~~N< 4,\\
	(2,+\infty)~~~\mbox{if}~~N =4,\\
		[N/2,+\infty)~~~\mbox{if}~~N>4.
	\end{cases}
\end{equation}
Under these conditions,  $\mathcal{L}_q$  with domain  $D(\mathcal{L}_q):=W^{2,2}(\Omega)\cap W^{1,2}_0(\Omega)$ defines a self-adjoint operator (see, e.g., \cite{edmund, Reed2}) so that its spectrum consists of an infinite sequence of eigenvalues $\{\lambda_i(q) \}_{i=1}^{\infty}$,  repeated according to their finite multiplicity and ordered as  $\lambda_1(q)<\lambda_2(q)\leq \ldots $.
Furthermore, the principal eigenvalue $\lambda_1(q)$ is a simple and isolated. 

The recover of the  potential
 $q(x)$ from a knowledge of the spectral data  $\{\lambda_i(q) \}_{i=1}^{\infty}$
is a classical problem and, beginning with the celebrated  papers by Ambartsumyan \cite{ambar} in 1929, Borg in 1946 \cite{borg}, Gel'fand \& Levitan \cite{gelL} in 1951, it received a lot of attention; see, e.g., surveys \cite{chadan, SavShk}. It is well known that a knowledge of the single spectrum $\{\lambda_i(q) \}_{i=1}^{\infty}$ is insufficient to determine the potential $q(x)$; see, e.g., \cite{borg, gelL}.

%


In this work  we deal with an inverse problem where  given finite set of eigenvalues:
$\{\lambda_i \}_{i=1}^{m}$, $m<+\infty$.
Having only finite spectral data, the inverse problem possesses infinitely many solutions. Thus additional conditions have to be imposed in order to make the problem well-posed. To overcome this difficulty, we  assume that an approximation $q_0$ of the potential $q$ is known. Under this assumption, it is natural to consider the following inverse optimization spectral problem: for a given $q_0$ and $\{\lambda_i \}_{i=1}^{m}$, $m<+\infty$, find a potential $\hat{q}$ closest to $q_0$ in a prescribed norm, such that $\lambda_i=\lambda_i(\hat{q})$ for all $i=1, \ldots, m$.

 In the present paper, we  study the following simplest variant of this problem:

\medskip

\par\noindent
$(P):$\,\,\textit{For a given $\lambda \in \mathbb{R}$ and $q_0 \in L^p(\Omega)$, find   a potential  $\hat{q} \in L^p(\Omega)$ such that  $\lambda=\lambda_1(\hat{q})$ and }
\begin{equation}\label{Var}
	\|q_0-\hat{q} \|_{L^p}=\inf\{||q_0-q||_{L^p}:~~ \lambda=\lambda_1(q), ~~q \in L^p(\Omega)\}.
\end{equation}

\medskip


It turns out that  this problem is related to the  following logistic nonlinear boundary value problem:
\begin{equation} \label{eq:Nonl}
\begin{cases}
-\Delta u+q_0 u=\lambda u- u^\frac{p+1}{p-1},~~~ x\in \Omega,	\\
~~u> 0,~~ x\in \Omega,	\\
~~u\bigr{|}_{\partial \Omega}=0.
\end{cases}
\end{equation}

Our first main result is as follows.

\begin{thm}\label{thm1}
Assume $\Omega$ is a bounded connected domain in $\mathbb{R}^N$ with a $C^{1,1}$-boundary $\partial \Omega$. Let $q_0 \in L^p(\Omega)$ be a given potential, where $p$ satisfies \eqref{Pas}. Then, for any $\lambda>\lambda_1(q_0)$,

$(1^o)$ there exists a unique potential   $\hat{q} \in L^p(\Omega)$ such that $\lambda=\lambda_1(\hat{q})$ and \eqref{Var} is satisfied;

$(2^o)$  there exists a weak positive solution $\hat{u} \in W^{1,2}_0(\Omega)$ of \eqref{eq:Nonl} such that 
$$
\hat{q}=q_0+\hat{u}^{2/(p-1)}~~\mbox{a.e. in}~~ \Omega.
$$
Furthermore, $\hat{u} \in C^{1, \beta}(\overline{\Omega})$ for some $\beta \in (0,1)$ and $\phi_1(\hat{q})=\hat{u}/\|\hat{u}\|_{L^2}$.
\end{thm}
Using the relationship between $(P)$ and \eqref{eq:Nonl} stated in Theorem \ref{thm1}, we are able to prove the following theorem on the uniqueness of the solution for \eqref{eq:Nonl}.
\begin{thm}\label{thm2}
Assume that
\begin{equation}\label{PasG}
\begin{cases}
	2<\gamma \leq 4~~~\mbox{if}~~N< 4,\\
	2<\gamma < 4~~~\mbox{if}~~N =4,\\
		2<\gamma\leq \frac{2N}{N-2}~~~\mbox{if}~~N>4.
	\end{cases}
\end{equation}
Then, for any $q_0 \in L^p(\Omega)$ with $p\geq \frac{\gamma}{\gamma-2}$ and any $\lambda>\lambda_1(q_0)$, the boundary value problem
\begin{equation} 
\tag{\ref{eq:Nonl}$'$}
\begin{cases}
-\Delta u+q_0 u=\lambda u- u^{\gamma-1},~~~ x\in \Omega,	\\
~~u\geq 0,~u \not\equiv 0,~~ x\in \Omega,	\\
~~u\bigr{|}_{\partial \Omega}=0,
\end{cases}
\end{equation}
 has at most one   weak solution.
\end{thm}

The existence of a solution for \eqref{eq:Nonl} follows in a standard way cf. \cite{brezis}. In the case  $q_0 \in L^\infty(\Omega)$, there are various proofs of the uniqueness of the solution for \eqref{eq:Nonl}; see, e.g., \cite{brezis, diaz} and  the references given there. 
However, as far as we know, the uniqueness in the case of an unbounded potential $q_0 \in L^p(\Omega)$ has not been proven before.


It should be emphasized that Theorem \ref{thm1}  also can be seen as a new method of proving 
the existence and uniqueness of a solution for nonlinear boundary value problems. Indeed, the finding of the 
minimizer $\hat{q}$ of constrained minimization problem \eqref{Var} also implies the existence of the solution $\hat{u}=(\hat{q}-q_0)^{(p-1)/2}$ for \eqref{eq:Nonl}, whereas the uniqueness of $\hat{u}$ follows from the uniqueness of the minimizer of  \eqref{Var}, as will be shown below.

This paper is organised as follows.
Section 2 contains some preliminaries. In Section 3, we give the proofs of Theorems \ref{thm1} and \ref{thm2}. In Section 4, using nonlinear problem \eqref{eq:Nonl}, we investigate stability properties of inverse optimization spectral problem $(P)$. Section 5 contains some remarks and open problems.

\section{Preliminaries}
In what follows, we denote by $\left\langle \cdot, \cdot \right\rangle $ and $\|\cdot\|_{L^2}$   the scalar product and the norm in  $L^2(\Omega)$, respectively; $W^{1,2}(\Omega), W^{2,2}(\Omega)$ are usual Sobolev spaces; $W^{1,2}_0:=W^{1,2}_0(\Omega)$ is the closure of $C^\infty_0(\Omega)$ in the norm 
$$
\|u\|_{1}=\left(\int_{\Omega} |\nabla u |^2 dx\right )^{1/2}.
$$
By a standard criterion (see,  e.g.,  \cite{edmund}, Theorem 1.4. p. 306), assumption \eqref{Pas} implies that $\mathcal{L}_q$ with domain $D(\mathcal{L}_q):=W^{2,2}(\Omega)\cap W^{1,2}_0(\Omega)$ is self-adjoint on $L^2(\Omega)$. Moreover, $\mathcal{L}_q$ is a semibounded operator  so that  the principal eigenvalue satisfies
\begin{equation}\label{lambda1}
	-\infty<\lambda_1(q)=\inf_{\phi \in W^{1,2}_0\setminus 0}\frac{\int_{\Omega} |\nabla \phi |^2 dx+\int_{\Omega} q\phi^2\,dx}{\int_{\Omega}\phi^2\,dx},
\end{equation}
where the minimum attained at  eigenfunction $\phi_1 \in W^{1,2}_0\setminus 0$. The  regularity of solutions for elliptic equations (see, e.g.,  Lemma B 3 in \cite{struw}) implies that $\phi_1 \in W^{2,q}(\Omega)$ for any $q\geq 2$ and therefore by the Sobolev  theorem, $\phi_1 \in C^{1,\alpha}(\overline{\Omega})$  for any $\alpha\in (0,1)$. Furthermore, in view of \eqref{Pas}, we may apply the weak Harnack inequality (see Theorem 5.2. in \cite{Trud}) and obtain, in a standard fashion (see, e.g., Theorem 8.38 in \cite{GilTrud} ),  that
the principal eigenvalue $\lambda_1(q)$ is simple and $\phi_1>0$ in $\Omega$.

By the Sobolev  theorem, we have continuous embeddings: $W^{2,2}_0(\Omega) \subset L^{\infty}(\Omega)$ if $N<4$, $W^{2,2}_0(\Omega) \subset L^{q}(\Omega)$, $\forall q \in [2,\infty)$ if $N=4$, $W^{2,2}_0(\Omega) \subset L^{2N/(N-4)}(\Omega)$ if $N>4$. Hence, by Holder's inequality and \eqref{Pas} 
\begin{equation}\label{22}
	\int_{\Omega} q^2\psi^2\,dx\leq a \|\Delta \psi \|_{L^2}^2+b\|\psi\|_{L^2}^2~~~~\forall \psi \in D(\mathcal{L}_q),
\end{equation}
for some constants $a,b \in (0,+\infty)$ which do not depend on $\psi \in D(\mathcal{L}_q)$. This implies  that for $q_0,q \in L^p$ and $\varepsilon \in \mathbb{R}$, the family $\mathcal{L}_{q_0+\varepsilon q}$ is analytic of type ($A$) (see \cite{Reed2}, p. 16) and therefore, by Theorem X.12 in \cite{Reed4},  $\mathcal{L}_{q_0+\varepsilon q}$ is an analytic family in the sense of Kato. 
Hence by  Theorem X.8 in \cite{Reed4}, $\lambda_1(q_0+\varepsilon q)$ is an analytic function of $\varepsilon$ near $0$ and $\phi_1(q_0+\varepsilon q)$ analytically depends on $\varepsilon$ near $0$ as a function of $\varepsilon$
with values in $L^2$.  
\begin{lemma}\label{lem1}
	$\lambda_1(q)$  is a continuously
differentiable map in $L^p$ with the Fr\'echet-derivative
\begin{equation}\label{eq:Val}
	D\lambda_1(q)(h)=\frac{1}{\|\phi_1(q)\|^2_{L^2}}\int_\Omega \phi_1^2(q) h\, dx, ~~\forall \,q,h \in L^p.
\end{equation}
\end{lemma}
\begin{proof}
	Set $\|\phi_1(q)\|_{L^2}=1$. Observe,
\begin{equation}\label{eq:deriv}
	\frac{d}{d \varepsilon} \lambda_1(q+\varepsilon h)|_{\varepsilon =0}=\int_\Omega \phi_1^2(q) h\, dx, ~~\forall \,q,h \in L^p.
\end{equation}
Indeed, testing equation $\mathcal{L}_{q+\varepsilon h} \phi_1(q+\varepsilon h)=\lambda_1(q+\epsilon h)\phi_1(q+\varepsilon h)$ by $\phi_1(q)$ and integrating by parts one obtains
\begin{align*}
	\int_{\Omega} \phi_1(q+\varepsilon h) (-\Delta \phi_1(q))\, dx + \int_{\Omega} &(q+\varepsilon h)\phi_1(q+\epsilon h)\phi_1(q) \,dx = \\
	&\lambda_1(q+\epsilon h) \int_{\Omega} \phi_1(q+\epsilon h)\phi_1(q) \,dx.
\end{align*}
By the above, all terms in this equality are differentiable with respect to $\varepsilon$. Thus we have
\begin{align*}
	-\int_{\Omega}\Delta \phi_1(q) \frac{d}{d \varepsilon}&\phi_1(q+\varepsilon h)|_{\varepsilon =0} \,dx + \int_{\Omega} h\phi_1^2(q)\, dx +\int_{\Omega} h\frac{d}{d \varepsilon}\phi_1(q+\varepsilon h)|_{\varepsilon =0}\phi_1(q) dx = \\
	&\frac{d}{d \varepsilon}\lambda_1(q+\varepsilon h)|_{\varepsilon =0} \int_{\Omega} \phi_1^2(q)\,dx+\lambda_1(q) \int_{\Omega} \frac{d}{d \varepsilon}\phi_1(q+\varepsilon h)|_{\varepsilon =0}\phi_1(q)\,dx.
\end{align*}
Since $\phi_1(q)$ is an eigenfunction of $\mathcal{L}_{q}$, this implies \eqref{eq:deriv}.

To conclude the proof of the lemma, it is sufficient to show that 
\begin{equation}\label{eq:ContW}
	\phi_1(\cdot) \in C(L^p; W^{1,2}_0(\Omega)).
\end{equation}
Indeed, assume \eqref{eq:ContW} is true. Since \eqref{Pas}, by the Sobolev theorem, 
the embedding $W^{1,2}(\Omega) \subset L^\frac{2p}{(p-1)}$ is continuous. Hence the map $\phi_1(\cdot): L^p \to L^\frac{2p}{(p-1)}\cap L^2$ is continuous and therefore the norm of Gateaux derivative $D\lambda_1(q)$ continuously depends on $q \in L^p$. This implies  that $\lambda_1(q)$ is continuously differentiable in $L^p$. 

To prove \eqref{eq:ContW}, let us  first show that $\lambda_1(q)$  defines a continuous map in 
$L^p$. Suppose, contrary to our claim, that there is a sequence $(q_n)$ such that $q_n \to q$ in $L^p$ as $n \to \infty$ and
$|\lambda_1(q_n)-\lambda_1(q)|>\epsilon$ for some $\epsilon >0$, $n=1,2,...$. 
Consider Rayleigh's quotient
\begin{equation}\label{eq:CFR}
	\lambda_1(q_n)\equiv R_{q_n}(\phi_1(q_n)):=\frac{\|\phi_1(q_n)\|_1^2+\int_0^1q_n\phi_1^2(q_n)\,dx}{\|\phi_1(q_n)\|_{L^2}^2}~~~n=1,2,....
\end{equation}
It is easily seen that
\begin{equation}\label{eq:CF}
	R_{q_n}(\phi_1(q)) \to R_{q}(\phi_1(q))=\lambda_1(q) \,\, \mbox{as}\,\, n \to \infty.
\end{equation}
Hence and  since $\lambda_1(q_n)=R_{q_n}(\phi_1(q_n))\leq R_{q_n}(\phi_1(q))$, $n=1,2,...$, we conclude that $\lambda_1(q_n)=R_{q_n}(\phi_1(q_n))<C_0<+\infty$, where $C_0<+\infty$ does not depend on $n=1,2,...$.  Due to homogeneity of  $R_{q_n}(\phi_1(q_n))$ we may assume that $\|\phi_1(q_n)\|_1^2=1$ for all $n$. Hence the Banach-Alaoglu and Sobolev
theorems imply that there is a subsequence, which we again denote by $(\phi_1(q_n))$, such that $\phi_1(q_n) \to \bar{\phi}$  as $n \to \infty$ weakly in $W^{1,2}$ and strongly in $L^q$ for $q\in [2,2^*)$, where $2^*=2N/(N-2)$ if $N>2$ and $2^*=+\infty$ if $N\leq 2$. Observe that if $\bar{\phi}=0$, then $\|\phi_1(q_n)\|_{L^2} \to 0$ and $\int_0^1q_n\phi_1^2(q_n) \to 0$ as $n \to \infty$, which implies by \eqref{eq:CFR} that $\lambda_1(q_n) \to +\infty$. We get a contradiction. Thus $\bar{\phi}\neq 0$ and therefore $|\lambda_1(q_n)|<C$, where $C<+\infty$ does not depend on $n=1,2,...$. 
Hence, in view of \eqref{eq:CF}, we conclude that 
$$
\lambda_1(q)=R_{q}(\phi_1(q))\leq R_{q}(\bar{\phi}) \leq \liminf_{n\to \infty}R_{q_n}(\phi_1(q_n))\leq \lim_{n\to \infty}R_{q_n}(\phi_1(q))=\lambda_1(q),
$$
which contradicts to our assumption. Thus, indeed, the map $\lambda_1(\cdot): L^p \to \mathbb{R}$ is continuous.

Take $q,q_0 \in L^p$. Then 
\begin{align*}
-\Delta(\phi_1(q_0)-\phi_1(q))+q_0(\phi_1(q_0)-\phi_1(q))-\lambda_1(q_0)(\phi_1(q_0)-\phi_1(q))+&\\
(q_0-q)\phi_1(q)-(\lambda_1(q_0)-\lambda_1(q))\phi_1(q)=0&.
\end{align*}
Testing this equation by $(\phi_1(q)-\phi_1(q_0))$ and integrating by parts, we obtain
\begin{align*}
&\|\phi_1(q_0)-\phi_1(q)\|_1^2+ \int_\Omega q_0(\phi_1(q_0)-\phi_1(q))^2\,dx-\lambda_1(q_0)\int_\Omega (\phi_1(q_0)-\phi_1(q))^2\,dx +\\
&\int_\Omega(q_0-q)\phi_1(q)(\phi_1(q_0)-\phi_1(q))\,dx-(\lambda_1(q_0)-\lambda_1(q))\int_\Omega \phi_1(q)(\phi_1(q_0)-\phi_1(q))\,dx=0.
\end{align*} 
Let $q_k \to q_0$ in $L^p$ as $k \to \infty$. We may assume that $\|\phi_1(q_k)\|_1=1$, $k=1,2,...$. Set $t_k:=\|\phi_1(q_k)-\phi_1(q_0)\|_1$, $\psi_k:=(\phi_1(q_k)-\phi_1(q_0))/t_k$, $k=1,2,...$. Then $0<t_k<1+\|\phi_1(q_0)\|_{1}:=C_1<+\infty$, $\|\psi_k\|_1=1$ and
 \begin{align}\label{eV:tk}
&t_k\left(1+ \int_\Omega q_0\psi_k^2\,dx-\lambda_1(q_0)\int_\Omega \psi_k^2\,dx\right) =\nonumber\\
&-\int_\Omega(q_0-q_k)\phi_1(q_k)\psi_k\,dx+(\lambda_1(q_0)-\lambda_1(q_k))\int_\Omega \phi_1(q_k)\psi_k\,dx,\,\, k=1,2,....
\end{align} 
By H\"older's inequality and the Sobolev theorem, we have
\begin{align*}
&
|\int_\Omega(q_k-q_0)\phi_1(q_k)\psi_k\,dx|\leq \|q_k-q_0\|_{L^p}\|\phi_1(q_k)\|_{L^\frac{2p}{(p-1)}}\|\psi_k\|_{L^\frac{2p}{(p-1)}}\leq C_2\|q_k-q_0\|_{L^p}, \\
&|(\lambda_1(q_k)-\lambda_1(q_0))\int_\Omega \phi_1(q_k)\psi_k\,dx|\leq 
 C_2|\lambda_1(q_k)-\lambda_1(q_0)|,
\end{align*} 
where $C_2<+\infty$ does not depend on $k=1,2,...$. Hence and from \eqref{eV:tk} it follows that $t_k:=\|\phi_1(q_k)-\phi_1(q_0)\|_1 \to 0$ as $q_k \to q_0$ in $L^p$. Thus we get \eqref{eq:ContW}.
\end{proof}

\begin{lemma}\label{lem2}
	 $\lambda_1(q)$ is strictly concave functional in $L^p$.
\end{lemma}
\begin{proof}
Let $q_1,q_2 \in L^p \setminus 0$. Denote $\phi_1^t:=\phi_1(t q_1+(1-t) q_2)$. Assume that $\|\phi_1^t\|_{L^2}=1$, $t \in [0,1]$. Then due to \eqref{lambda1} we have
\begin{align*}
\lambda_1(t q_1+(1-t) q_2)=&	\int_{\Omega}|\nabla \phi_1^t|^2\, dx+\int_{\Omega}(t q_1+(1-t) q_2)|\phi_1^t|^2\, dx =\\
& t( \int_{\Omega}|\nabla \phi_1^t|^2\, dx+\int_{\Omega}q_1|\phi_1^t|^2\, dx) +(1-t)( \int_{\Omega}|\nabla \phi_1^t|^2\, dx+\int_{\Omega}q_2|\phi_1^t|^2\, dx) >\\
&t\lambda_1( q_1) +(1-t)\lambda(q_2),~~\forall t \in (0,1),
\end{align*}
which yields the proof.
\end{proof}

\section{Proof of the main results}

{\it Proof of Theorem \ref{thm1}.}

Let $q_0 \in L^p$ and $\lambda>\lambda_1(q_0)$. Consider the constrained minimization problem
\begin{equation}\label{MinP}
\hat{Q}=\min\{Q(q):q \in M_{\lambda}\},
\end{equation}
where $Q(q):=||q_0-q||^p_{L^p}$ for $q \in L^p(\Omega)$ and
$$
M_{\lambda}:=\{q \in L^p(\Omega):~~\lambda\leq \lambda_1(q)\}.
$$
Notice that $M_{\lambda} \neq \emptyset$. Indeed, $\lambda=\lambda_1(q_0+\lambda-\lambda_1(q_0))$, $\forall \lambda \in \mathbb{R}$ and thus   $q_0+\lambda-\lambda_1(q_0) \in M_{\lambda}$ for any $\lambda>\lambda_1(q_0)$.  Moreover, by Lemma \ref{lem2}, $M_{\lambda}$ is convex. Hence, by coerciveness of $Q: L^p \to \mathbb{R}$ there exists a minimizer $\hat{q} \in M_{\lambda}$ of \eqref{MinP}. Since the strong inequality $\lambda>\lambda_1(q_0)$, it follows that $\hat{q} \neq 0$. The convexity of $M_{\lambda}$ and $Q$ entails that $\hat{q}$ is unique and that 
$$
\hat{q} \in \partial M_{\lambda}= \{q\in L^p:~~\lambda=\lambda_1(q)\}.
$$
This concludes the proof of assertion $(1^o)$ of Theorem \ref{thm1}.

Evidently $Q$ is $C^1$-functional in $L^p$. Hence, in view of Lemma \ref{lem1}, the Lagrange multiplier rule implies 
\begin{equation}
	\mu_1 DQ(\hat{q})(h)+\mu_2D\lambda_1(\hat{q})(h) =0,~~ \forall h \in L^p,
\end{equation}
where $\mu_1,\mu_2 $ such that $|\mu_1|+|\mu_2|\neq 0$, $\mu_1 \geq 0$, $\mu_2 \leq 0$. Thus by \eqref{eq:Val} we deduce
\begin{equation}
	\int_\Omega (-\mu_1 p (q_0-\hat{q})|q_0-\hat{q}|^{p-2}+\mu_2\phi_1^2(\hat{q})) h\, dx  =0, \,\, \forall h \in L^p,
\end{equation}
where $\|\phi_1^2(\hat{q})\|_{L^2}=1$. Arguing by contradiction, it is easily to conclude that
 $\mu_1 > 0,\mu_2< 0$. Thus we have 
$$
(q_0-\hat{q})|q_0-\hat{q}|^{p-2}=\mu \phi_1^2(\hat{q})~~\mbox{a.e. in}~~\Omega,
$$
where $\mu=\mu_2/(p \mu_1)<0$. Notice that $\phi_1>0$ in $\Omega$. Hence $q_0<\hat{q}$ a.e. in $\Omega$ and 
\begin{equation}
	\hat{q}=q_0+\nu \phi_1^{2/(p-1)}(\hat{q})~~\mbox{a.e. in}~~\Omega,
\end{equation}
where $\nu:=(-\mu)^{1/(p-1)}>0$.
Substituting this into \eqref{eq:S} yields
\begin{equation}\label{eq:phi}
- \Delta \phi_1(\hat{q})+q_0 \phi_1(\hat{q})={\lambda} \phi_1(\hat{q})-\nu\phi_1^\frac{p+1}{p-1}(\hat{q}).
\end{equation}
Thus, indeed, $\hat{u}=\nu^\frac{p-1}{2}\phi_1(\hat{q})$ satisfies \eqref{eq:Nonl}.  Moreover,  $\hat{q}=q_0+\hat{u}^{2/(p-1)}$ a.e. in $\Omega$. This concludes the proof of $(2^o)$.

\medskip

\noindent
{\it Proof of Theorem \ref{thm2}.}
Since  (\ref{eq:Nonl}$'$) is obtained from \eqref{eq:Nonl}  by replacing $\gamma=\frac{2p}{p-1}$, it is sufficient to prove the uniqueness of the solution for \eqref{eq:Nonl}.

First we prove 
\begin{lemma}\label{lem:Nonl}
Let $u \in W^{1,2}_0$ be a nonnegative weak solution of \eqref{eq:Nonl}. Then the function $\bar{q}:=q_0+u^{2/(p-1)}$ is a local minimum point of $Q$  in $M_{\lambda}$. 
\end{lemma}
\begin{proof}
Let $u  \in W^{1,2}_0$ be a nonnegative weak solution of \eqref{eq:Nonl}.  The  regularity solutions for elliptic equations (see, e.g.,  Lemma B 3 in \cite{struw}) implies that $u \in W^{2,q}(\Omega)$ for any $q\geq 2$ and therefore  $u \in C^{1,\alpha}(\overline{\Omega})$  for any $\alpha\in (0,1)$. By the weak Harnack inequality (see Theorem 5.2. in \cite{Trud}) it follows that $u>0$ in $\Omega$. This implies that $u$ is an eigenfunction of $\mathcal{L}_{\bar{q}}$ corresponding to the principal eigenvalue $\lambda =\lambda_1(\bar{q})$, i.e., $u=\phi_1(\bar{q})$. In view of Lemma \ref{lem1}, by the Lusternik theorem \cite{lust}  the tangent space of $\partial M_\lambda$ at $ q \in \partial M_\lambda$ can be expressed as follows
\begin{equation}\label{Tangent}
	T_{q}(\partial M_\lambda):=\{h \in L^p:~D\lambda_1(q)( h)\equiv \int_\Omega u^2\cdot h\, dx = 0\}.
\end{equation}
From this 
$$
D_qQ(q)|_{q=q_0+u^{2/(p-1)}}(h)=p\int_\Omega u^2\cdot h\, dx = 0, ~~\forall h \in T_{\bar{q}}(\partial M_\lambda).
$$ 
Hence and since
$$
D_{qq}Q(q)|_{q=q_0+u^{2/(p-1)}}(h,h)=p(p-1)\int_\Omega u^\frac{2(p-2)}{p-1}\cdot h^2\, dx>0,~~\forall h \in T_{q}(\partial M_\lambda),
$$
we obtain that 
$$
Q(\bar{q}+h)> Q(\bar{q}),
$$ 
for any  $h \in T_{\bar{q}}(\partial M_\lambda)$ with sufficient small $\|h\|_{L^p}$. 
\end{proof}

Let us conclude the proof of Theorem \ref{thm2}.
By Theorem \ref{thm1} we know that \eqref{eq:Nonl} possess a solution $\hat{u}$ such that the functional $Q$ admits a global minimum at $\hat{q}=q_0+\hat{u}^{2/(p-1)}$ on $M_{\lambda}$. Assume that there exists a second weak solution $\bar{u}$ of \eqref{eq:Nonl}. Then by Lemma \ref{lem:Nonl}, $\bar{q}=q_0+\bar{u}^{2/(p-1)}$ is a  local minimum point of $Q$  in $M_{\lambda}$. However, due to strict convexity of $\lambda_1(q)$ and $Q(v)$ this is possible  only if 
 $\bar{q}=\hat{u}$.

\section{Stability results}
In this section, we prove that the solution $\hat{q}$ of the inverse optimization spectral problem $(P)$ is stable with respect to variation of $q_0$ and $\lambda$.

Let $q_0 \in L^p$, $\lambda >\lambda_1(q_0)$. Denote by $\hat{q}(\lambda, q_0)$ the unique solution of $(P)$ obtained by Theorem \ref{thm1} and denote by $\hat{u}(\lambda, q_0)$ the corresponding solution of \eqref{eq:Nonl}. 

\begin{proposition}\label{Prop1}
\begin{description}
	\item[(i)] For any $\lambda \in (\lambda_1(q_0),+\infty)$, 
	the map $\hat{u}(\lambda, \cdot):  L^p \to W^{1,2}_0$ is continuous and thus $\hat{q}(\lambda, q_0)$ continuously depends on $q_0$ as a map from $L^p$ to $L^p$.
	\item[(ii)] For any $q_0 \in L^p$, 
	the map $\hat{u}(\cdot, q_0):  (\lambda_1(q_0),+\infty) \to W^{1,2}_0$ is continuous and thus $\hat{q}(\lambda,q_0)$ continuously depends on $q_0$ as a map from $L^p$ to $L^p$. Moreover, $\hat{q}(\lambda,q_0) \to q_0$ in  $L^p$ as $\lambda \downarrow \lambda_1(q_0)$.
\end{description}
\end{proposition}
\begin{proof}
	 We give the proof of \textbf{(i)} only for the case $N\geq 3$; the other cases are left to the
reader. Let $\lambda \in (\lambda_1(q_0),+\infty)$. Assume $q_n \to q_0$ in $L^p$ as $n\to \infty$. By the above,  $\lambda_1(q)$ continuously depends on $q\in L^p$. Thus for sufficiently large $n$ we have $\lambda>\lambda_1(q_n)$.

We claim that the sequence $\|\hat{u}(\lambda, q_n)\|_1$, $n=1,2,...$ is bounded and separated from zero. 
Set $t_n:=\|\hat{u}(\lambda, q_n)\|_1 $,  $v_n:=\hat{u}(\lambda, q_n)/t_n$, $n=1,2,...$.
Since $\|v_n\|_1=1$,  $n=1,2,...$, by the Banach-Alaoglu and Sobolev
theorems we may assume that $v_n \rightharpoondown v$ weakly in $W^{1,2}$, $v_n \to v$ a.e. in $\Omega$ and strongly in $L^q$, $2\leq q <2N/(N-2)$ for some $v \in W^{1,2}$.	 It follows from \eqref{eq:Nonl}
\begin{equation}\label{eq:st1}
	\|v_n\|_1+\int_\Omega q_n v_n^2\,dx-\lambda \int_\Omega v_n^2\,dx+t_n^\frac{2}{p-1}\int_\Omega v_n^\frac{2p}{p-1}\,dx=0, ~~n=1,2,....
\end{equation}
 By H\"older's inequality
	$$
	|\int_\Omega q_n v_n^2\,dx|\leq \|q_n\|_{L^p}\|v_n\|^2_{L^{2p/(p-1)}}, ~~n=1,2,..., 
	$$
	where, in view of \eqref{Pas}, we have $2<2p/(p-1)<2N/(N-2)$. Suppose that $v=0$. Then 
$$
\|v_n\|_1+\int_\Omega q_n v_n^2\,dx-\lambda \int_\Omega v_n^2\,dx+t_n^\frac{2}{p-1}\int_\Omega v_n^\frac{2p}{p-1}\,dx\geq \|v_n\|_1+\int_\Omega q_n v_n^2\,dx-\lambda \int_\Omega v_n^2\,dx \to 1,
$$	
as $n\to +\infty$, which contradicts to  \eqref{eq:st1}. Thus $v \neq 0$ and therefore by \eqref{eq:st1} the sequence $t_n$ is bounded. Assume, by contradiction, that  $t_n \to 0$. Then passing to the limit in \eqref{eq:Nonl} yields
$$
-\Delta v+q_0 v=\lambda v.
$$
From the above, it follows that $v\geq 0$, $v\neq 0$. Thus $v$ is an eigenfunction corresponding to the principal eigenvalue of $\mathcal{L}_{q_0}$. However, by the assumption $\lambda>\lambda_1(q_0)$ and we get a contradiction.  

Thus the claim is proving and we may assume that $\hat{u}(\lambda, q_n) \rightharpoondown \bar{u}$ weakly in $W^{1,2}$, $\hat{u}(\lambda, q_n) \to \bar{u}$ a.e. in $\Omega$ and strongly in $L^q$, $2\leq q <2N/(N-2)$ for some $\bar{u} \in W^{1,2}\setminus 0$.  Since $\hat{u}(\lambda, q_n)>0$ in $\Omega$, we conclude that $\bar{u} \geq 0$ a.e. in $\Omega$. Passing to the limit in \eqref{eq:Nonl} we obtain
$$
-\Delta \bar{u}+q_0 \bar{u}=\lambda \bar{u}- \bar{u}^\frac{p+1}{p-1}.
$$
Due to the uniqueness of solution of \eqref{eq:Nonl}, it follows that $\bar{u}=\hat{u}(\lambda, q_0)$. Furthermore, this implies that $\hat{u}(\lambda, q_n) \to \hat{u}(\lambda, q_0)$ strongly in $W^{1,2}$. Thus we have proved that the map $\hat{u}(\lambda, \cdot):  L^p \to W^{1,2}_0$ is continuous. 

Under assumption \eqref{Pas}, we have a  continuous embedding $W^{1,2} \subset L^{2p/(p-1)}$. Hence   $\hat{u}^{2/(p-1)}(\lambda, q_n) \to \hat{u}^{2/(p-1)}(\lambda, q_0)$  strongly in $L^p$ and thus $\hat{q}(\lambda, q_n)=q_0+\hat{u}^{2/(p-1)}(\lambda, q_n)$ strongly  converges to $\hat{q}(\lambda, q_0)=q_0+\hat{u}^{2/(p-1)}(\lambda, q_0)$ in $L^p$ as $n\to +\infty$. This concludes the proof of \textbf{(i)}.

The proof for the first part of \textbf{(ii)} is similar to \textbf{(i)}. To prove that $\hat{q}(q_0, \lambda) \to q_0$ in  $L^p$ as $\lambda \downarrow \lambda_1(q_0)$, it is remained to show that 
for $\lambda=\lambda_1(q_0)$  problem \eqref{eq:Nonl} may has only zero solution. Suppose, contrary to our claim, that there exists a positive solution $u$ of \eqref{eq:Nonl} for $\lambda=\lambda_1(q_0)$. Then testing the equation in \eqref{eq:Nonl} by $\phi_1(q_0)$ and integrating by parts we obtain
\begin{align*}
	\int_{\Omega} u (-\Delta \phi_1(q_0))\, dx + \int_{\Omega} q_0 u\phi_1(q_0) \,dx = 
	\lambda_1(q_0) \int_{\Omega} u\phi_1(q_0) \,dx- \int_{\Omega} u^\frac{p+1}{p-1}\phi_1(q_0) \,dx,
\end{align*}
which implies that $\int_{\Omega} u^\frac{p+1}{p-1}\phi_1(q) \,dx=0$. However this is possible only if $u\equiv 0$ in $\Omega$.
\end{proof}

%
%

\section{Conclusion remarks and open problems}\label{sec:conclusion}

Notice that if $\lambda<\lambda_1(q_0)$, then nonlinear boundary value problem has no solution (see, e.g., \cite{brezis}). However, the existence of solution of $(P)$ in the case $\lambda<\lambda_1(q_0)$ is unknown.

Since for various $p$ equation \eqref{eq:Nonl} has different solutions,  the answer on the inverse optimization spectral problem $(P)$ essentially depends on the prescribed  norm $\|\cdot\|_{L^p}$. We are unable to offer criteria necessary to identify preferred norms. However, it should be noted that a similar problem about choosing a suitable norm has already been encountered in the literature on the theory of inverse problems (see, e.g., \cite{Mar, SavShk}). 


It is an open problem to solve the  $m$-parametric  inverse optimization spectral problem
\par\noindent
$(P^m)$:\,\,\textit{Given $\lambda_1,...,\lambda_m \in \mathbb{R}$ and $q_0 \in L^p(\Omega)$. Find   a potential  $\hat{q} \in L^p(\Omega)$ such that  $\lambda_i=\lambda_i(\hat{q})$, $i=1,...,m$ and }
\begin{equation*}
	\|q_0-\hat{q} \|_{L^p}=\inf\{||q_0-q||_{L^p}:~~ \lambda_i=\lambda_i(q), ~i=1,...,m,~q \in L^p(\Omega)\}.
\end{equation*}
The above inference of the nonlinear boundary value problem \eqref{eq:Nonl} can be formally generalized to be applicable to problem $(P^m)$. In that case, one can obtain the following system of nonlinear equations
%

\begin{equation}\label{sys}
	\begin{cases}
		-\Delta u_i+q_0 u_i=\lambda_i u_i- (\sum_{j=1}^m \mu_j u_j^{2})^\frac{p}{p-1}u_i,~~i=1,2,...,m,\\ 
		~~u_i|_{\partial \Omega}=0,~~~~i=1,2,...,m.
	\end{cases}
\end{equation}
where $\mu_i\geq 0$, $i=1,2,...,m$ are some constants so that 
\begin{equation}\label{OPEN}
	\hat{q}=q_0+(\sum_{j=1}^m \mu_j u_j^{2})^\frac{p}{p-1}.
\end{equation}
However, we do not know how to justify this approach. Moreover, as far as we know, the existence and uniqueness of solution for \eqref{sys} with $q_0 \in L^p$ is also an open problem. Nevertheless, it would be useful to verify \eqref{OPEN}  numerically.

\end{document}